\documentclass[11pt]{amsart}

\usepackage{graphicx,pinlabel,bbm,amssymb}
\usepackage{todonotes}
\definecolor{darkblue}{rgb}{0,0,0.4} 
\usepackage[colorlinks=true, citecolor=darkblue, filecolor=darkblue, linkcolor=darkblue,urlcolor=darkblue]{hyperref}

\newtheorem{theorem}{Theorem}[section]

\newtheorem{proposition}[theorem]{Proposition}

\newtheorem{question}[theorem]{Question}

\newtheoremstyle{defn}{12pt}{12pt}{}{}{\bfseries}{.}{ }{}
\theoremstyle{defn}
\newtheorem{definition}[theorem]{Definition}

\theoremstyle{remark}

\newtheorem{remark}[theorem]{Remark}

\input xy
\xyoption{all}

\def\Z{\mathbb{Z}}
\def\ZZ{\mathbb{Z}}
\def\R{\mathbb{R}}

\def\A{\mathcal{A}}
\def\k{\mathbbm{k}}
\def\d{\partial}
\def\e{\epsilon}
\def\tb{\operatorname{tb}}
\def\LCH{\operatorname{LCH}}

\def\Aug{\operatorname{Aug}}
\def\Hom{\operatorname{Hom}}

\newcommand{\co}{\nobreak\mskip2mu\mathpunct{}\nonscript
  \mkern-\thinmuskip{:}\penalty300\mskip6muplus1mu\relax}

\begin{document}

\title{Torsion in linearized contact homology for Legendrian knots}
\author{Robert Lipshitz}
\thanks{\texttt{RL was supported by NSF Grant DMS-2204214 and a Simons Fellowship}}
\email{\href{mailto:lipshitz@uoregon.edu}{lipshitz@uoregon.edu}}
\address{Department of Mathematics, University of Oregon, Eugene, OR 97403}

\author{Lenhard Ng}
\thanks{\texttt{LN was supported by NSF Grant DMS-2003404.}}
\email{\href{mailto:ng@math.duke.edu}{ng@math.duke.edu}}
\address{Department of Mathematics, Duke University, Durham, NC 27708}

\begin{abstract}
  We present examples of Legendrian knots in
  $\mathbb{R}^3$ that have linearized Legendrian contact homology over
  $\Z$ containing torsion. As a consequence, we show that there
  exist augmentations of Legendrian knots over $\mathbb{Z}$ that
  are not induced by exact Lagrangian fillings, even though their
  mod $2$ reductions are.
\end{abstract}

\maketitle

\section{Introduction}
Holomorphic-curve invariants are powerful tools for studying Legendrian submanifolds of contact manifolds. This paper concerns the basic setting of Legendrian knots in $\R^3$, with respect to the standard contact structure $\ker (dz-y\,dx)$, and the holomorphic-curve invariant known as Legendrian contact homology.

In \cite{Che}, Chekanov, inspired by ideas of Eliashberg, associated a differential graded algebra over $\Z/2$ to a Legendrian knot $\Lambda\subset \R^3$; this DGA is called the \emph{Chekanov--Eliashberg DGA} and its homology is \emph{Legendrian contact homology}. (As is fairly common, we will use these terms largely interchangeably; the homology itself plays no role in this paper.)
Up to the appropriate notion of equivalence---stable tame isomorphism---the Chekanov--Eliashberg DGA is invariant under Legendrian isotopy, but it is difficult to tell if two DGAs are stable tame isomorphic. To extract more tractable invariants, Chekanov used augmentations (homomorphisms to $\Z/2$) to construct finite-dimensional chain complexes from the DGA. The homologies of these complexes are called \emph{linearized (Legendrian) contact homology} and denoted $\LCH_*^\e(\Lambda;\Z/2)$; Chekanov showed that the collection of linearized contact homologies over all augmentations $\e$ is a Legendrian-isotopy invariant. Using this, he gave the first example of a pair of Legendrian knots, of topological type $m(5_2)$, which have the same classical invariants (knot type, Thurston--Bennequin number, and rotation number) but are not Legendrian isotopic.

Since Chekanov's work, the study of Legendrian contact homology has developed into a rich subject with connections to areas including microlocal sheaf theory, cluster algebras, homological mirror symmetry, and topological string theory. We refer the reader to the survey \cite{ENsurvey} for a discussion of Legendrian contact homology in $\R^3$, but we briefly mention a couple of developments that are relevant here. 

First, although Chekanov introduced augmentations and linearized contact homology as purely algebraic objects associated to the Chekanov--Eliashberg DGA, we now understand certain augmentations as having geometric origins. Specifically, in~\cite{EHK}, Ekholm--Honda--K\'alm\'an showed that any exact Lagrangian filling $L$ of a Legendrian knot $\Lambda$ induces an augmentation of the DGA of $\Lambda$. Furthermore, the linearized contact homology associated to this augmentation is precisely the usual homology of the filling $L$; this result is colloquially called the Seidel isomorphism.

Second, the coefficient ring of the Chekanov--Eliashberg DGA can be lifted from $\Z/2$ to $\Z$ by assigning coherent orientations to the Floer-type moduli spaces underlying Legendrian contact homology. This was done first for Legendrian knots in $\R^3$ in \cite{ENS} and then for more general Legendrian submanifolds in arbitrary dimension in \cite{EES-ori}. One can subsequently tensor with an arbitrary field $\k$ and study Legendrian contact homology over $\k$, as a number of papers have done, or study Legendrian contact homology over $\Z$ directly.

With $\Z$ coefficients, linearized contact homology is lifted from a graded $\Z/2$-vector space to a graded $\Z$-module. This raises the possibility that it might contain torsion. For Legendrians in high dimension (where the contact manifold has dimension $\geq 5$), it is well-known that linearized contact homology can indeed contain torsion. The earliest examples of torsion in high dimension were provided in \cite{EES-ori}, where torsion is used to distinguish between Legendrian submanifolds that share the same classical invariants and Legendrian contact homology over $\Z/2$. As another example, torsion for knot conormal tori has been shown to encode the determinant of the underlying smooth knot (see~\cite{NgKCH2}). More recently, in \cite{Golovko} Golovko showed that any finitely generated abelian group can appear as the linearized contact homology of some high-dimensional Legendrian.

In this paper, we show that torsion in linearized contact homology also appears for Legendrian knots in $\R^3$. To our knowledge, it was previously an open question whether such torsion could exist. Throughout this paper, we use $\LCH_*^\e(\Lambda)$ to denote the linearized contact homology of $\Lambda$ over $\Z$ associated to a $\Z$-valued augmentation $\e$. 

\begin{proposition}
  There are Legendrian knots $\Lambda$ in $\R^3$ such that for any $n\geq 2$, there is a $\ZZ$-valued augmentation $\e_n \co \A_\Lambda \to \Z$
  of the Chekanov--Eliashberg DGA of $\Lambda$ for which the linearized contact homology $\LCH_*^{\e_n}(\Lambda)$ contains a $\Z/n$ summand.
\label{prop:main}
\end{proposition}

\noindent
To prove Proposition~\ref{prop:main}, we give a particular Legendrian knot, of topological type $m(8_{21})$, for which torsion exists for specific augmentations (Section~\ref{ssec:821-ex}). Afterwards, we generalize this example to larger families of knots which have the same property.

By the universal coefficient theorem, Proposition~\ref{prop:main} implies that the linearized contact cohomology for these augmentations (which is more directly connected to sheaf theory; cf.~\cite{STZ,NRSSZ}) also has torsion.

As a consequence of the proof of Proposition~\ref{prop:main}, we also establish an analogue of Golovko's result for Legendrians in $\R^3$.

\begin{proposition}\label{prop:geography}
For any finitely generated abelian group $G$ and any $k \in \Z$ with $k\neq 0,1$, there is a Legendrian knot $\Lambda$ in $\R^3$ and a $\Z$-valued augmentation $\e$ of the Chekanov--Eliashberg DGA of $\Lambda$ such that $\LCH_k^\e(\Lambda) \cong G$.
\end{proposition}

\noindent
Here the index $k$ is the grading on Legendrian contact homology; see Section~\ref{ssec:background}. The exclusion $k\neq 0,1$ is necessary because Sabloff duality \cite{Sabloff:duality,EES:duality} imposes constraints on $\LCH_k^\e(\Lambda)$ for $k=0,1$. Proposition~\ref{prop:geography} is proven in Section~\ref{sssec:Lambdak-linhom}.

We also use our torsion examples to explore the question of when augmentations are geometric, in the sense that they are induced by an exact Lagrangian filling. 

\begin{proposition}
For the family of Legendrian knots $\Lambda_k$ defined in Section~\ref{ssec:family}, there are augmentations to $\Z$ that are not geometric, even though their mod $2$ reductions to $\Z/2$ are geometric.
\label{prop:geom}
\end{proposition}

\noindent
The obstruction to augmentations being geometric is provided by a version of the Seidel isomorphism established by Gao and Rutherford \cite{GR}.

We conclude the paper by discussing which knots in the Legendrian knot atlas \cite{atlas} do and do not have torsion, and some speculation based on these observations.

\begin{remark}
  In this paper, we consider contact homology linearized with respect to
  a single augmentation. One can also linearize with respect to two
  different augmentations. In terms of the augmentation
  category of~\cite{BC:aug-cat,CDGG:aug-cat,NRSSZ}, these
  bilinearized contact homologies are the morphism spaces between
  different objects, while the linearized contact homologies considered in this
  paper are the endomorphism spaces of a single object. While writing
  this paper, we learned that Fr\'ed\'eric Bourgeois and Salammbo Connolly have completely solved
  the geography problem for bilinearized contact homologies in
  dimension $3$---that is, which graded homology groups can appear as the bilinearized contact homology of some Legendrian knot with respect to two augmentations.  They also obtain some partial results in the linearized case, so that there is some overlap between Bourgeois--Connolly’s results and ours.
\end{remark}
\subsection*{Acknowledgments} We thank Fr\'ed\'eric Bourgeois, Dan
Rutherford, and Josh Sabloff for helpful conversations, and the referee for many further suggestions. 
The results in
this paper were observed while working on a related project with
Sucharit Sarkar, and we are particularly grateful to him for his
insights and suggestions. The first author conducted this research
while visiting Stanford University and, briefly, UCLA, and thanks both
for their hospitality.

\section{Results}

\subsection{Legendrian contact homology, augmentations, and linearized homology}
\label{ssec:background}

We start with a brief review of Legendrian contact homology, mostly to fix notation and conventions. For more details, see e.g.\ \cite{ENsurvey}.

Throughout this paper, we consider only Legendrian knots in $\bigl(\R^3,\xi_{\mathit{std}}=\ker(dz-y\,dx)\bigr)$.
We will usually represent a Legendrian knot $\Lambda$ by its \emph{front projection} $\Pi_{xz}(\Lambda)$ in $\R^2_{xz}$. Assuming that $\Lambda$ is generic, the front projection $\Pi_{xz}(\Lambda)$ has only two types of singularities: transverse double points (\emph{crossings}) and semicubical cusps (\emph{left cusps} or \emph{right cusps} depending on whether the cusp is a local minimum or local maximum for the $x$ coordinate, respectively). It is customary to omit the crossing information at double points of $\Pi_{xz}$, but this can be recovered from the front projection by noting that the strand with lower slope crosses over the strand with higher slope.

All knots in this paper will have rotation number $0$; in terms of the front projection, this means that $\Pi_{xz}(\Lambda)$ has an equal number of up cusps and down cusps (cusps traversed upwards and downwards, respectively, with respect to some orientation of $\Lambda$). We will also restrict our attention to knots whose front projections are \textit{simple} in the sense of \cite{CLI}: this assumption means that all right cusps share the same $x$ coordinate, and makes it easier to compute the differential for Legendrian contact homology.

The Chekanov--Eliashberg DGA is most naturally defined in terms of the \textit{Lagrangian projection} $\Pi_{xy}$ of a Legendrian knot $\Lambda$. Crossings of the knot diagram $\Pi_{xy}(\Lambda)$ are the \emph{Reeb chords} of $\Lambda$ (flows of the Reeb vector field $\partial/\partial z$ with endpoints on $\Lambda$) and generate the DGA. In order to instead produce a DGA associated to a front, we use a procedure called \emph{resolution} from~\cite{CLI} that takes the front projection of $\Lambda$ and turns it into the Lagrangian projection of a knot that is Legendrian isotopic to $\Lambda$. Resolution smooths out left cusps, replaces double points of the front by crossings, and replaces right cusps by loops with a single crossing. There is a one-to-one correspondence between crossings and right cusps of the front of $\Lambda$ and Reeb chords of the resolved diagram, and we will sometimes accordingly abuse language and use ``Reeb chords'' to mean the crossings and right cusps of $\Pi_{xz}(\Lambda)$.

Label the Reeb chords of the front $\Pi_{xz}(\Lambda)$ by $a_1,\ldots,a_n$, and place a basepoint somewhere along $\Pi_{xz}(\Lambda)$ away from the crossings; we introduce one more indeterminate $t$ corresponding to this basepoint. We then construct the free unital $\Z$-algebra
\[
\A_\Lambda = \Z\langle a_1,\ldots,a_n,t^{\pm 1}\rangle,
\]
generated by $a_1,\ldots,a_n,t^{\pm 1}$, with no relations except for $t\cdot t^{-1} = t^{-1}\cdot t = 1$. The algebra $\A_\Lambda$ is graded by setting $|t|=|t^{-1}|=0$ and assigning a particular grading $|a_i| \in \Z$ to each $a_i$: if $a_i$ is a right cusp then $|a_i|=1$, while if $a_i$ is a crossing then $|a_i|$ is the difference between the number of up cusps and down cusps traversed by a path along $\Pi_{xz}(\Lambda)$ from the undercrossing of $a_i$ (the strand of the crossing with higher slope) to the overcrossing of $a_i$ (the strand with lower slope). The fact that these gradings are well-defined and $\Z$-valued uses our assumption that the rotation number vanishes. 

For the definition of the differential, see e.g.~\cite{CLI} or~\cite{ENsurvey}. Since signs are important for our computations, we note that our sign conventions follow those references: the sign of a term in $\d_\Lambda$ is $(-1)$ raised to the number of corners of the disk that occur at a crossing of even degree and occupy the downward-facing quadrant at that crossing. 
We will usually place the basepoint at some right cusp $a_i$ of $\Pi_{xz}(\Lambda)$, in which case 
the only place $t$ appears in $\d a_j$ for any $j$ is in $\d a_i$: $\d a_i = t + \cdots$. (The sign of $t$ here chooses an orientation for $\Lambda$, but our results will be unchanged if we reverse the orientation instead.)

An $R$-valued augmentation of $\A_\Lambda$ is a DGA map
\[
\e \co \A_\Lambda \to R
\]
where $R$ is a unital commutative ring lying in degree $0$, with trivial differential. (In this paper, $R$ will usually be $\Z$ or a field $\k$.) Specifically, this means that $\e(1)=1$, $\e(a) = 0$ for any $a\in\A_\Lambda$ of nonzero degree, and $\e\circ\d = 0$.

A fundamental result of Leverson states that for any augmentation $\e$ valued in a field $\k$, $\e(t)=-1$; see~\cite{Leverson}. (This uses the fact that $\Lambda$ is a knot, not a more general link.) It follows that the same holds for $\Z$-valued augmentations, as well.

There is a standard procedure that starts from $(\A_\Lambda,\d)$ equipped with an augmentation and produces a finitely generated complex of free $R$-modules, whose homology is defined to be the linearized contact homology $\LCH_*^{\e}(\Lambda;R)$ of $\Lambda$ with respect to the augmentation $\e$; see e.g.\ \cite{ENsurvey}. Here we present a slight variant that may be useful for computations. Define
\[
\A^R = (R[s]) \langle a_1,\ldots,a_n\rangle = ((\A_\Lambda \otimes \Z[s]) \otimes R)/(t=\e(t))
\]
where $s$ should be viewed as a parameter. There is an $R$-algebra map $\phi^\e \co \A^R \to \A^R$ defined by
$\phi^\e(1)=1$, $\phi^\e(s) = s$, and 
\[
\phi^\e(a_i) = s a_i + \e(a_i)
\]
for all $i$. It follows from the fact that $\e\circ\d=0$ that $(\phi^\e(\d a_i))|_{s=0} = 0$ for all $i$. Define
\[
\d^\e a_i = \left.\frac{d}{ds}\right|_{s=0} \phi^\e(\d a_i).
\]
By construction, $\d^\e a_i$ is linear in $a_1,\ldots,a_n$ for each $i$. Thus,
\[
\d^\e \co V \to V,
\]
where $V$ is the free graded $R$-module generated by $a_1,\ldots,a_n$, and
$(\d^\e)^2=0$ since $\d^2=0$. We now define the linearized contact homology of $\Lambda$ with respect to $\e$ to be
\[
\LCH_*^{\e}(\Lambda;R) = H_*(V,\d^\e).
\]
When $R=\Z$, we will suppress it from this notation.

A fundamental property of linearized contact homology is \emph{Sabloff duality}, a relationship between linearized contact homology and its dual akin to Poincar\'e duality for ordinary homology. Specifically, given a Legendrian $\Lambda$ in $\R^3$ and a $\k$-valued augmentation $\e$ for $\Lambda$ for some field $\k$, Sabloff duality states that there are non-canonical isomorphisms
\begin{align}
  \LCH_i^\e(\Lambda;\k)&\cong \LCH_{-i}^\e(\Lambda;\k) & i\neq 1\label{eq:Sabloff-generic}\\
  \LCH_1^\e(\Lambda;\k)&\cong \LCH_{-1}^\e(\Lambda;\k)\oplus\k\label{eq:Sabloff-1}.
\end{align}
(A canonical version of the first line identifies $\LCH_i^\e$ with the dual of $\LCH_{-i}^\e$, and the second becomes a short exact sequence.) This was first proved in the case $\k=\Z/2$ by Sabloff, in~\cite{Sabloff:duality}, and then extended to higher dimensions and more general rings by Ekholm--Etnyre--Sabloff, in~\cite{EES:duality}. The form we have used here requires an understanding of what they call the \emph{manifold classes}, which follows either from Sabloff's original argument or from Remark 5.6 of~\cite{EES:duality}.

\subsection{Torsion for \texorpdfstring{$m(8_{21})$}{the 8,21 knot}}
\label{ssec:821}

In this section we consider the Legendrian knot $\Lambda_0$ shown in Figure~\ref{fig:821}. This knot is of topological type $m(8_{21})$ and was studied by Melvin and Shrestha, who showed in \cite{MS} that it has two different linearized contact homologies corresponding to different augmentations over $\Z/2$. (The depiction of $\Lambda_0$ in \cite[Figure~2]{MS} differs from ours by reflection in the $z$ axis, but the two are Legendrian isotopic: the contactomorphism $(x,y,z) \mapsto (-x,-y,z)$, which in the front projection is reflection in the $z$ axis, is isotopic to the identity map through contactomorphisms.)

\subsubsection{Linearized contact homology for $m(8_{21})$}
\label{ssec:821-ex}

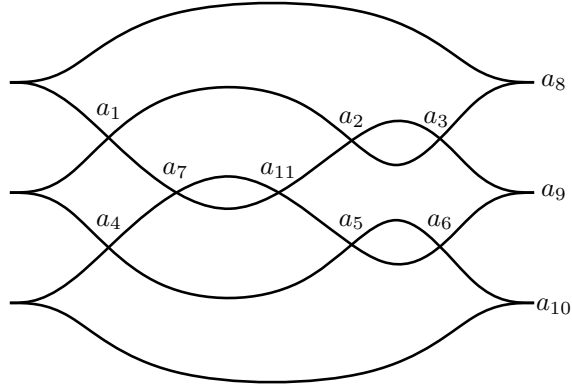
\begin{figure}
\labellist
\small\hair 2pt
\pinlabel $a_1$ at 27 74
\pinlabel $a_2$ at 93 72
\pinlabel $a_3$ at 116 72
\pinlabel $a_4$ at 27 44
\pinlabel $a_5$ at 93 44
\pinlabel $a_6$ at 117 44
\pinlabel $a_7$ at 45 58
\pinlabel $a_8$ at 148 82
\pinlabel $a_9$ at 148 52
\pinlabel $a_{10}$ at 148 21
\pinlabel $a_{11}$ at 73 58
\endlabellist
\centering
\includegraphics[height=2in]{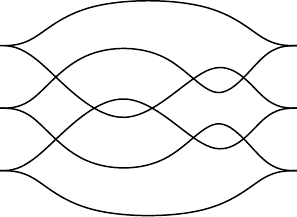}
\caption{
\textbf{The knot $\Lambda_0$.} This knot appears as $m(8_{21})$ in the Legendrian knot atlas, and has this topological knot type.}
\label{fig:821}
\end{figure}

The crossings and right cusps of $\Lambda_0$ are labeled in Figure~\ref{fig:821}, and are graded as follows:
\begin{align*}
|a_7|=|a_8|=|a_9|=|a_{10}| &= 1 \\
|a_1|=|a_2|=|a_3|=|a_4|=|a_5|=|a_6| &= 0 \\
|a_{11}| &= -1.
\end{align*}
We place a basepoint at the cusp $a_8$.

The differential $\d$ on the Chekanov--Eliashberg DGA $\A_{\Lambda_0}$ is:
\begin{gather*}
  \begin{aligned}
\d a_2 &= a_4 a_{11}\qquad & \d a_8 &= t+a_1+a_3+a_1a_2a_3+a_7a_{11}a_3 \\
\d a_5 &= -a_{11} a_1 & \d a_9 &= 1-(1+a_3a_2)a_1a_6-a_3(a_4+a_6+a_4a_5a_6) \\
\d a_7 &= -a_1a_4 &
\d a_{10} &= 1-a_4-a_6-a_6a_5a_4-a_6a_{11}a_7
\end{aligned}\\
\d a_1 = \d a_3 = \d a_4 = \d a_6 = \d a_{11} = \d t = \d t^{-1} = 0.
\end{gather*}

Let $R$ be a (unital) integral domain; of interest to us are the cases where $R$ is $\Z$ or a field. A (graded) augmentation $\e \co \A_{\Lambda_0} \to R$ is determined by the $7$-tuple $(\e(a_1),\ldots,\e(a_6),\e(t)) \in R^7$. By solving the system of equations that arise from setting $\epsilon\circ\d = 0$ ($\e(a_1)\e(a_4)=-\e(\d a_7)=0$, etc.), we find that
$\e$ is an augmentation if and only if one of the following two lines holds:
\begin{align*}
\e(t) &= -1 & \e(a_4) &= 0 &\quad \e(a_6) &= 1 & \quad \e(a_1)+\e(a_3)+\e(a_1)\e(a_2)\e(a_3)&= 1 \\
\e(t) &= -1 & \e(a_1) &= 0 &\quad \e(a_3) &= 1 & \quad \e(a_4)+\e(a_6)+\e(a_4)\e(a_5)\e(a_6)&= 1.
\end{align*}

In particular, for any $n\in\Z$, there is an augmentation $\e_n \co \A_{\Lambda_0} \to \Z$ defined by
\[
\e_n(a_1,a_2,a_3,a_4,a_5,a_6,t) = (n,-1,1,0,0,1,-1).
\]
If we write $V$ for the graded $\Z$-module generated by $a_1,\ldots,a_{11}$, then $\d$ and $\e_n$ induce a linear differential $\d^{\e_n} \co V \to V$ as described in Section~\ref{ssec:background}, and the homology of $(V,\d^{\e_n})$ is $\LCH_*^{\e_n}(\Lambda_0)$.
Specifically, we have
\begin{gather*}
\d^{\e_n} a_5 = -n a_{11} \qquad \d^{\e_n} a_7 = -n a_4 \qquad \d^{\e_n} a_8 = na_2-(n-1)a_3 \\
\d^{\e_n} a_9 = -a_4-a_6-na_2+(n-1)a_3 \qquad \d^{\e_n} a_{10} = -a_4-a_6 \\
\d^{\e_n} a_1 = \d^{\e_n} a_2 = \d^{\e_n} a_3 = \d^{\e_n} a_4 = \d^{\e_n} a_6 = \d^{\e_n} a_{11} = 0
\end{gather*}
and thus
\[
\LCH_*^{\e_n}(\Lambda_0) \cong \begin{cases}
\Z & *=1 \\
\Z^2 \oplus \Z/n & *=0 \\
\Z/n & *=-1 \\
0 & \text{otherwise.}
\end{cases}
\]
This proves Proposition~\ref{prop:main}.

\subsubsection{Geometric motivation}
\label{sssec:821geom}

Here we discuss a geometric reason why the existence of torsion for $m(8_{21})$, which was initially surprising to the authors, could have been anticipated.

For a general Legendrian knot $\Lambda$ and any field $\k$, the set of augmentations from the DGA $\A_{\Lambda}$ to $\k$ forms a variety over $\k$, the \textit{augmentation variety} $\Aug(\Lambda,\k)$. As mentioned above, by \cite{Leverson}, any augmentation $\e$ must satisfy $\e(t) = -1$. Thus, if $a_1,\ldots,a_\ell$ are the degree-$0$ Reeb chords of $\Lambda$, an augmentation is uniquely determined by $(\e(a_1),\ldots,\e(a_\ell))$, and we can view
\[
\Aug(\Lambda,\k) \subset \k^\ell.
\]
This is a variety because it is the vanishing set of a collection of polynomials in $a_1,\ldots,a_\ell$ given by $\d a_i$ where $a_i$ ranges over all degree-$1$ generators of $\A_\Lambda$.

Given an augmentation $\e$ viewed as a point in $\Aug(\Lambda,\k)$, we can consider the Zariski tangent space to the augmentation variety at the point $\e$, $T_\e\Aug(\Lambda,\k)$. It is an exercise in algebra to check that
\[
T_\e\Aug(\Lambda,\k) \cong \ker((\d^\e)^* \co A_0^* \to A_1^*)
\]
where $A_0$ and $A_1$ are the $\k$-vector spaces generated by the degree $0$ and degree $1$ Reeb chords of $\Lambda$, respectively. In the special case when $\Lambda$ has no degree $-1$ Reeb chords, $A_{-1}=0$ and so $T_\e\Aug(\Lambda,\k)$ is isomorphic to $\LCH^0_\e(\Lambda;\k)$, the degree-$0$ linearized contact cohomology of $\Lambda$ with respect to $\e$ with coefficients in $\k$. Even when $\Lambda$ does have degree $-1$ Reeb chords, there is a surjection $T_\e\Aug(\Lambda,\k) \twoheadrightarrow \LCH^0_\e(\Lambda;\k)$. We summarize this by the heuristic ``the larger the Zariski tangent space at $\e$, the larger the linearized contact (co)homology'': this slogan can be made more precise but will suffice for our purposes. Note that the homology $\LCH_0^\e(\Lambda;\k)$ and cohomology $\LCH^0_\e(\Lambda;\k)$ are isomorphic by the universal coefficient theorem.

In the case that $\Lambda=\Lambda_0$, if we use the obvious coordinates $a_1,\ldots,a_6$ on $\k^6$, then
\[
\Aug(\Lambda_0,\k) = V_1 \cup V_2 \subset \k^6
\]
where
\begin{align*}
V_1 &= \{a_4=0,~a_6=1,~a_1+a_3+a_1a_2a_3 = 1\} \\
V_2 &= \{a_1=0,~a_3=1,~a_4+a_6+a_4a_5a_6=1\}.
\end{align*}
Note that $V_1$ and $V_2$ are smooth $3$-dimensional subvarieties of $\Aug(\Lambda_0,\k)$ that intersect in the $2$-dimensional subvariety $V_1\cap V_2 = \{a_1=0,~a_3=1,~a_4=0,~a_6=1\}$. The Zariski tangent space to a point $\e$ in $\Aug(\Lambda_0,\k)$ is larger for $\e \in V_1\cap V_2$ (where it has dimension $4$) than for $\e \in V_1\setminus V_2$ or $\e\in V_2 \setminus V_1$ (where it has dimension $3$). Consistent with the heuristic above, this is borne out in linearized contact homology: one can calculate that
\[
\LCH_0^\e(\Lambda_0;\k) \cong \begin{cases} \k^4 & \e\in V_1\cap V_2 \\
\k^2 & \e\in (V_1\setminus V_2) \cup (V_2\setminus V_1).
\end{cases}
\]

Consider the augmentation $\e_n \co \A_{\Lambda_0} \to \Z$ described in Section~\ref{ssec:821-ex}. For any prime $p$, we can compose $\e_n$ with the projection $\Z \to \Z/p$ to get the mod $p$ reduction $\e_{n;p} \co \A_{\Lambda_0} \to \Z/p$. Then $\e_{n;p}$ is in the component $V_1$ of $\Aug(\Lambda_0,\Z/p)$ for all $p$, and is also in $V_2$ if and only if $p\,|\,n$. Thus,
\begin{equation}\label{eq:LCH-mod-p}
\LCH_0^{\e_{n;p}}(\Lambda_0;\Z/p) \cong \begin{cases} (\Z/p)^4 & p\,|\,n \\
(\Z/p)^2 & p\nmid n.
\end{cases}
\end{equation}
However, the complex whose homology yields $\LCH_*^{\e_{n;p}}(\Lambda_0;\Z/p)$ is precisely the tensor product of the complex whose homology yields $\LCH_*^{\e_n}(\Lambda_0)$ with $\Z/p$. By the universal coefficient theorem, Formula~\eqref{eq:LCH-mod-p} forces either $\LCH_0^{\e_n}(\Lambda_0)$ or $\LCH_{-1}^{\e_n}(\Lambda_0)$ to have $p$-torsion if $p\,|\,n$; and indeed both of these groups have $p$-torsion.

\subsection{Torsion for \texorpdfstring{the family $\Lambda_k$}{an infinite family}}
\label{ssec:family}

The knot $\Lambda_0$ is the smallest of a family of Legendrian knots
with augmentations whose linearized contact homology contains
torsion. Here we describe the rest of the family.

For $k \geq 1$, let $\Lambda_k$ denote the Legendrian knot shown in Figure~\ref{fig:945-family}. The case $k=0$ is exactly the $m(8_{21})$ knot considered in Section~\ref{ssec:821}; however, restricting to $k \geq 1$ actually makes the computation of linearized homology slightly simpler. (The knot $\Lambda_1$ is the second $m(9_{45})$ knot in the Legendrian knot
atlas~\cite{atlas}.)

\begin{figure}
\labellist
\small\hair 2pt
\pinlabel $a_1$ at 27 86
\pinlabel $a_2$ at 94 84
\pinlabel $a_3$ at 115 83
\pinlabel $a_4$ at 26 56
\pinlabel $a_5$ at 44 70
\pinlabel $a_6$ at 68 67
\pinlabel $a_7$ at 85 50
\pinlabel $a_8$ at 147 93
\pinlabel $a_9$ at 147 55
\pinlabel $a_{10}$ at 126 54
\pinlabel $a_{k+10}$ at 124 13
\pinlabel $a_{k+11}$ at 152 40
\pinlabel $a_{2k+11}$ at 153 10
\pinlabel $k$ at 103 34
\endlabellist
\centering
\includegraphics[height=2.5in]{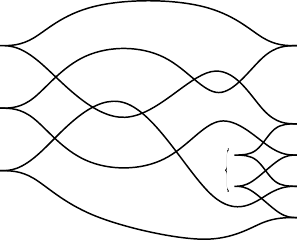}
\caption{
\textbf{The knot $\Lambda_k$ for $k \geq 1$.} The depicted knot is $\Lambda_2$. The Reeb chords $a_{10},\dots,a_{k+10}$ correspond to the crossings on the vertical segment from $a_{10}$ to $a_{k+10}$, and the chords $a_{k+11},\dots,a_{2k+11}$ correspond to the right cusps on the vertical segment from $a_{k+11}$ to $a_{2k+11}$.}
\label{fig:945-family}
\end{figure}
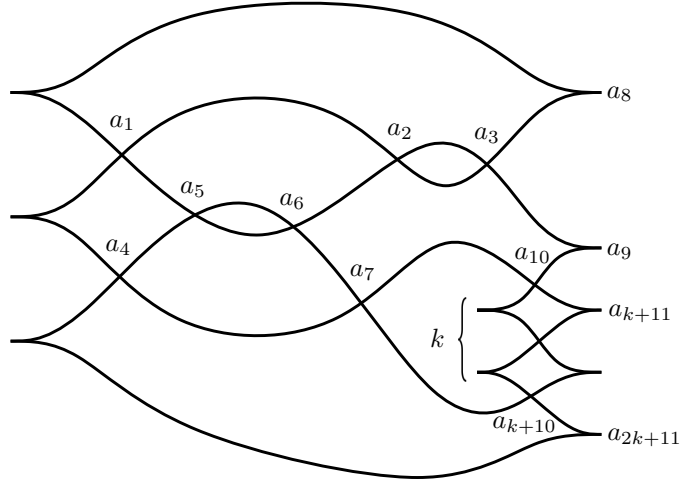

\subsubsection{Linearized homology for $\Lambda_k$}
\label{sssec:Lambdak-linhom}

The Chekanov--Eliashberg DGA of $\Lambda_k$ is $(\A_{\Lambda_k},\d)$, where $\A_{\Lambda_k} = \Z\langle a_1,\ldots,a_{2k+11},t^{\pm 1}\rangle$ and the grading on $\A_{\Lambda_k}$ is given by
\begin{align*}
|a_5| &= k+1 \\
|a_4| &= k \\
|a_8| = |a_9| = |a_{k+11}| = \cdots=|a_{2k+11}| &= 1\\
|a_1|=|a_2|=|a_3|=|a_{10}|=\cdots=|a_{k+10}|=|t^{\pm 1}| &= 0\\
|a_7| &= -k \\
|a_6| &= -k-1.
\end{align*}
The differential is nonzero on the following generators:
\begin{align*}
\d a_5 &= -a_1a_4 \\
\d a_7 &= (-1)^{k+1}a_6a_1 \\
\d a_8 &= t+a_1+a_3+a_1a_2a_3+a_{5}a_{6}a_3 \\
\d a_9 &= 1-(a_1+a_3+a_3a_2a_1+a_3a_4a_{7})a_{10} \\
\d a_{k+11+i} &= 1-a_{10+i}a_{11+i} \qquad \text{for $0\leq i\leq k-1$} \\
\d a_{2k+11} &= 1-a_{k+10}(1+a_6a_5+a_7a_4),
\end{align*}
where we place the basepoint at the cusp $a_8$.

From this, it is easy to check that a graded algebra map $\e \co \A_{\Lambda_k} \to R$ is an augmentation if and only if
$\e(t)=-1$, $\e(a_{10}) = \cdots = \e(a_{k+10})=1$, and 
\[
\e(a_1)+\e(a_3)+\e(a_1)\e(a_2)\e(a_3) =1.
\]
In particular, we can define an augmentation $\e_n$ for any $n\in\Z$ by $\e_n(a_1) = n$, $\e_n(a_2) = -1$, $\e_n(a_3) = 1$. Then the linearized differentials $\d^{\e_n}a_5 = -na_4$ and $\d^{\e_n}a_{7} = (-1)^{k+1}na_6$ produce $n$-torsion in linearized contact homology. To be precise:
\begin{align*}
\LCH_*^{\e_n}(\Lambda_1) &\cong \begin{cases}
\Z \oplus (\Z/n) & *=1 \\
\Z^2 & *=0 \\
\Z/n & *=-2 \\
0 & \text{otherwise}
\end{cases}\\
\LCH_*^{\e_n}(\Lambda_k) &\cong \begin{cases}
\Z & *=1 \\
\Z^2 & *=0 \\
\Z/n & *=k \text{ or } *=-k-1 \\
0 & \text{otherwise}
\end{cases} & k>1.
\end{align*}

\begin{figure}
  \centering
  \includegraphics{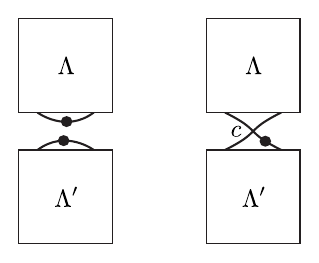}
  \caption{\textbf{The connected sum.} Left: two front
    diagrams. Right: their connected sum. (Compare~\cite[Figure
    2]{EHsum}.) The dots indicate our choice of basepoints; the new
    crossing is called $c$.}
  \label{fig:conn-sum}
\end{figure}

Proposition~\ref{prop:geography}, that any finitely-generated abelian
group can be obtained as a linearized contact homology group in any
grading $\neq 0,1$, now follows easily by using connected sums as in Melvin--Shrestha's
paper~\cite{MS}. Given Legendrian knots $\Lambda$ and $\Lambda'$, we
can form their connected sum $\Lambda\#\Lambda'$ as in
Figure~\ref{fig:conn-sum}. Most Reeb chords for $\Lambda\#\Lambda'$
are either Reeb chords for $\Lambda$ or $\Lambda'$; there is one
additional chord $c$ with $|c|=0$. If we choose basepoints as indicated in Figure~\ref{fig:conn-sum} (rather than at right cusps), the
differential on $\Lambda\#\Lambda'$ is induced from the differentials
on $\Lambda$ and $\Lambda'$ in an easy-to-describe manner: replace $t$ by $c$ in the
differential on $\A_\Lambda$, replace $t'$ by $-tc$ in the
differential on $\A_{\Lambda'}$, and set $\d(c)=0$. Hence,
augmentations for $\Lambda\#\Lambda'$ (sending $t$ to $-1$) correspond
to pairs of augmentations for $\Lambda$ and $\Lambda'$ (sending $t,t',c$ to $-1$), and with respect to this correspondence,
\[
\LCH_i^{\epsilon\#\epsilon'}(\Lambda\#\Lambda')\cong
\LCH_{i}^{\epsilon}(\Lambda)\oplus\LCH_{i}^{\epsilon'}(\Lambda')
\]
for all $i\neq 0,1$. (One can also analyze the behavior in gradings
$0$ and $1$ using Sabloff duality, as in~\cite{MS}.) Thus, to obtain a group
$G=\ZZ^m\oplus \ZZ/n_1\oplus\cdots\oplus\ZZ/n_k$ in grading $i>1$, we
simply take the connect sum of $\Lambda_i$ with itself $m+k$ times,
with augmentation
\[
\overbrace{\epsilon_0\#\cdots\#\epsilon_0}^m\#\epsilon_{n_1}\#\cdots\#\epsilon_{n_k}.
\]

\subsubsection{Geometric motivation}

As in Section~\ref{sssec:821geom} for $m(8_{21})$, one can interpret torsion for $\Lambda_k$ in terms of the augmentation variety, but the interpretation is slightly different for $k\geq 1$ than for $m(8_{21})$. Over a field $\k$, the augmentation variety of $\Lambda_k$ is
\[
\Aug(\Lambda_k,\k) = \{a_1+a_3+a_1a_2a_3=1\} \subset \k^3.
\]

Unlike for $\Lambda_0$, this variety is smooth. However, the linearized homology at an augmentation $\e \in \Aug(\Lambda_k,\k)$ still depends on the point $\e$:
\[
\LCH_*^{\e}(\Lambda_k;\k) \cong \begin{cases}
\k_{k+1}\oplus \k_k\oplus \k_1 \oplus \k_0^2 \oplus \k_{-k} \oplus \k_{-k-1} &
\e(a_1) = 0 \\
\k_1 \oplus \k_0^2 & \e(a_1) \neq 0, \end{cases}
\]
where subscripts denote grading. Thus given a $\Z$-valued augmentation $\e$ with $\e(a_1) = n$ and a prime $p$, the linearized homology over $\Z/p$ at the mod $p$ reduction of $\e$ is larger if $p\,|\,n$ than if $p\nmid n$. As before, it follows from the universal coefficient theorem that $\LCH_*^{\e}(\Lambda_k)$ contains $p$-torsion for any prime $p$ dividing $n$, in line with our calculation in Section~\ref{sssec:Lambdak-linhom}.

\subsubsection{A more general family with torsion}

The family $\Lambda_k$ is itself part of a larger family of Legendrian knots with augmentations whose linearized homology contains torsion. Here we sketch this family.

Suppose that we have two Legendrian knots $\Lambda'$ and $\Lambda''$ with fronts as shown on the left of Figure~\ref{fig:clasp-family}, and $x$ and $y$ are the indicated crossings. Further, suppose that the DGAs $\A_{\Lambda'}$ and $\A_{\Lambda''}$ have $\Z$-valued augmentations $\e'$ and $\e''$ satisfying the following conditions:
\begin{enumerate}
\item
$\e'(x)=n$ for some $n\not\in\{-1,0,1\}$ (so in particular, $|x|=0$);
\item
$\e''(y) = 0$; and
\item $y$ does not appear as a term in the linearized differential $\d(c)$ for any Reeb chord $c$ of $\Lambda''$.
\end{enumerate}

\begin{figure}
\labellist
\small\hair 2pt
\pinlabel $\Lambda'$ at -6 84
\pinlabel $\Lambda''$ at -6 27
\pinlabel $\Lambda$ at 140 56
\pinlabel $x$ at 23 89
\pinlabel $y$ at 23 32
\pinlabel $x$ at 168 68
\pinlabel $y$ at 168 44
\pinlabel $v$ at 173 61
\pinlabel $w$ at 190 61
\pinlabel $z$ at 290 56
\endlabellist
\centering
\includegraphics[width=\textwidth]{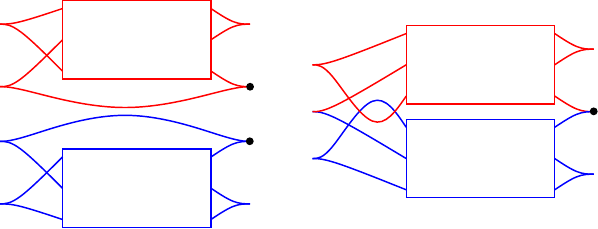}
\caption{
\textbf{The knots $\Lambda'$, $\Lambda''$, and $\Lambda$.} Inside the rectangles, the knots $\Lambda',\Lambda''$ are arbitrary, subject to the condition that they have augmentations $\e',\e''$ with the specified properties. Base points are placed as shown.}
\label{fig:clasp-family}
\end{figure}
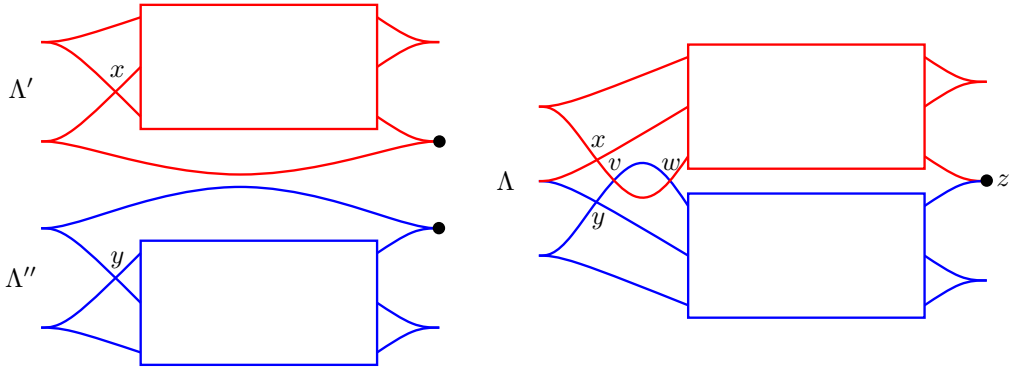

Construct the Legendrian knot $\Lambda$ shown on the right of Figure~\ref{fig:clasp-family}; this is the connected sum of $\Lambda'$ and $\Lambda''$, but with an additional clasp. For example, if $\Lambda'$ and $\Lambda''$ are the trefoil and twist knot shown in Figure~\ref{fig:trefoil-twist}, then $\Lambda$ is the knot $\Lambda_k$ considered earlier.

\begin{figure}
\labellist
\small\hair 2pt
\pinlabel $k$ at 117 18
\endlabellist
\centering
\includegraphics[width=3.5in]{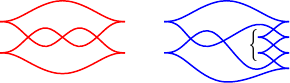}
\caption{
\textbf{The knots inducing the family $\Lambda=\Lambda_k$.} The knot $\Lambda'$ is the trefoil (left) while $\Lambda''$ is the twist knot (right).}
\label{fig:trefoil-twist}
\end{figure}

Let $(\A_{\Lambda},\d_\Lambda)$ denote the DGA of $\Lambda$. 
We can construct an algebra map $\e \co \A_\Lambda \to \Z$ by combining $\e'$ and $\e''$, as follows. Each Reeb chord $a$ of $\Lambda$ corresponds to a Reeb chord of either $\Lambda'$ or $\Lambda''$, except for the chords labeled $v,w,z$ in Figure~\ref{fig:clasp-family}. Define $\e(t)=-1$, $\e(v)=\e(w)=\e(z)=0$, and for all other Reeb chords $a$ of $\Lambda$, define $\e(a)$ to be either $\e'(a)$ or $\e''(a)$, depending on whether $a$ is a Reeb chord of $\Lambda'$ or $\Lambda''$. It can readily be checked that $\e$ is an augmentation of $(\A_\Lambda,\d_\Lambda)$: $\e\circ\d_\Lambda = 0$. Indeed, when applied to chords of $\Lambda$ coming from $\Lambda'$ or $\Lambda''$ (i.e., all chords besides $z,v,w$), $\e\circ\d_\Lambda$ agrees with $\e\circ\d$ on $\A_{\Lambda'}$ or $\A_{\Lambda''}$. We have $\e\circ\d_\Lambda(v)=\e\circ\d_\Lambda(w)=0$ since $\d_\Lambda(v) = -xy$ and $\d_\Lambda(w) = 0$. Finally, $\e\circ\d_\Lambda(z) = 0$: the only contributions to $\e\circ\d_\Lambda(z)$ come from terms in $\d_\Lambda(z)$ not involving any of $v$, $w$, or $y$, and these terms correspond to the product of the terms (other than $t$) in the differentials of the bottom right-cusp of $\Lambda'$ and the top right-cusp of $\Lambda''$ (with no $y$ factor). 

By the given conditions on $\e'$ and $\e''$, we find that $\d^\e(v) = -ny$, and further that this is the only place $y$ appears in $\d^\e(a)$ for any Reeb chord $a$ of $\Lambda$. It follows that $\LCH_*^\e(\Lambda)$ contains a $\Z/n$ summand generated by $y$.

\subsection{Geometric augmentations over \texorpdfstring{$\Z$}{the integers} and \texorpdfstring{$\Z/2$}{the field with two elements}}
\label{ssec:geom}

When a Legendrian knot has an exact, embedded Lagrangian filling, the filling induces an augmentation (more precisely, a family of augmentations) of the Legendrian knot. Here we investigate whether particular augmentations of the Legendrian knot $\Lambda_k$ introduced in Section~\ref{ssec:family} come from a filling, and prove Proposition~\ref{prop:geom}.

Recall that an \textit{exact Lagrangian filling} of a Legendrian knot $\Lambda$ is a Lagrangian surface $L$ in the symplectization $\bigl(\R\times\R^3,d(e^t(dz-y\,dx))\bigr)$ such that, for some sufficiently large $T$, $L \cap \bigl((T,\infty)\times\R^3\bigr) = (T,\infty)\times\Lambda$, equipped with a function $f\co L\to\R$ with $e^t(dz-y\,dx)|_L = df$. In this paper, all of our fillings will be embedded and orientable; henceforth we use the word ``filling'' to denote an exact, embedded, orientable Lagrangian filling.

Contact homology behaves functorially with respect to cobordisms. In the setting of fillings, this result can be stated as follows.

\begin{proposition}[\cite{EHK,Karlsson}]\label{prop:filling}
  A filling $L$ of a Legendrian knot $\Lambda$ induces a DGA map
  $\e_L \co \A_\Lambda \to \Z[H_1(L)]$ (where the right side has trivial differential).
\end{proposition}

\begin{remark}
For a filling $L$, the DGA map $\e_L$ from Proposition~\ref{prop:filling} may only preserve the $\Z/2$ grading on $\A_\Lambda$ induced from the $\Z$ grading; it is only guaranteed to preserve the full $\Z$ grading if $L$ has Maslov number $0$. See e.g.\ \cite{ENsurvey} for further discussion. \label{rmk:Maslov}
\end{remark}

If we tensor by a field $\k$, we obtain a DGA map $\A_\Lambda \to \k[H_1(L)]$. We can obtain a $\k$-valued augmentation of $\Lambda$ by composing with a homomorphism $\k[H_1(L)] \to \k$; this is equivalent to choosing a rank $1$ local system on $L$, i.e., a group homomorphism $H_1(L) \to \k^\times$. The same construction works over $\ZZ$, where by a local system we simply mean a homomorphism $H_1(L)\to\{\pm 1\}$. Note that when $\k=\Z/2$, there is a unique rank $1$ local system on $L$ and the filling $L$ produces a unique augmentation of $\Lambda$.

\begin{definition}
A $\Z$-valued augmentation $\e \co \A_\Lambda \to \Z$ of $\Lambda$ is \textit{geometric} if there is a filling $L$ of $\Lambda$ and a group homomorphism $H_1(L) \to \{\pm 1\}$ such that $\e$ is equal to the composition
\[
\A_{\Lambda} \stackrel{\e_L}{\longrightarrow} \Z[H_1(L)] \longrightarrow \Z.
\]
Similarly, if $\k$ is a field, then a $\k$-valued augmentation is \textit{geometric} if there is a group homomorphism $H_1(L) \to \k^\times$ such that $\e$ is equal to the composition 
\[
\A_\Lambda \stackrel{\e_L}{\longrightarrow} \Z[H_1(L)] \longrightarrow \k.
\]
\end{definition}

It is known that not all $\k$-valued augmentations of Legendrian knots are geometric, even for $\k=\Z/2$; see e.g.\ \cite{Chantraine-cobordism,EHK,Etgu,GR}. However, less is known for $\Z$-valued augmentations. The rest of this section is devoted to proving:

\begin{proposition}
For any $k \geq 1$ and $n\in\Z$, let $\Lambda_k$ and $\e_n$ be the Legendrian knot and $\Z$-valued augmentation defined in Section~\ref{ssec:family}. If $n$ is odd and $n\neq\pm 1$, then $\e_n : \A_{\Lambda_k}\to \Z$ is not geometric, but the mod $2$ reduction of $\e_n$ is geometric.
\label{prop:geometric}
\end{proposition}

To prove Proposition~\ref{prop:geometric}, we first show that $\e_n$ is not geometric, and then that its mod $2$ reduction is geometric. For the former, we use the following result, called the \emph{Seidel isomorphism}, versions of which are due to many authors (e.g.,~\cite{EkholmSFT,DR,Karlsson,GR}). The variant that we will use is due to Gao and Rutherford \cite{GR}.

\begin{proposition}[{\cite[Proposition 3.4]{GR}}]
Let $\Lambda$ be a Legendrian knot and let $\e \co \A_\Lambda \to \k$ be a geometric augmentation induced by a filling $L$ of $\Lambda$ equipped with a rank $1$ local system. Then there is a $\Z/2$-graded isomorphism
\label{prop:seidel}
\[
LCH_*^\e(\Lambda;\k) \cong H_{*+1}(L,\Lambda;\k).
\]
\end{proposition}

\begin{remark}
Proposition~\ref{prop:seidel} is actually a special case of \cite[Proposition~3.4]{GR}. More precisely, that result states that if $\e_L$ is a $\k$-valued augmentation induced by a filling $L$, then
\[
H^*\Hom_+(\e_L,\e_L) \cong H^*(L;\k).
\]
See \cite{NRSSZ} for the definition of $\Hom_+$, but to deduce Proposition~\ref{prop:seidel}, it suffices to note that
$H^*\Hom_+(\e_L,\e_L) \cong LCH_{1-*}^\e(\Lambda;\k)$ by Sabloff duality (see \cite[\S 5]{NRSSZ}), while $H^*(L;\k) \cong H_{2-*}(L,\Lambda;\k)$ by Poincar\'e duality. Also note that the fact that the isomorphism is only $\Z/2$-graded in general is because the Maslov number of $L$ might not be $0$; cf.\ Remark~\ref{rmk:Maslov}.
\end{remark}

Now consider the augmentation $\e_n \co \A_{\Lambda_k} \to \Z$ for $n\neq 0,\pm 1$. If this were geometric, then its mod $p$ reduction $\e_{n;p} \co \A_{\Lambda_k} \to \Z/p$ would also be geometric for any prime $p$. But it follows from the calculation of $\LCH_*^{\e_n}(\Lambda_k)$ from Section~\ref{ssec:family} that if $p$ is a prime dividing $n$, then $\LCH_*^{\e_{n;p}}(\Lambda_k;\Z/p)$ has dimension $7$ as a vector space over $\Z/p$. On the other hand, by \cite{Chantraine-cobordism}, any filling $L$ of $\Lambda_k$ must satisfy $2g(L)-1 = \tb(\Lambda_k) = 1$, whence $L$ is a punctured torus and $H_*(L,\Lambda;\k)$ has total dimension $3$. This contradicts Proposition~\ref{prop:seidel}, and we conclude that $\e_n$ is not geometric.

To complete the proof of Proposition~\ref{prop:geometric}, we will show that the mod $2$ reduction $\e_{n;2} \co \A_{\Lambda_k} \to \Z/2$ is geometric when $n$ is odd. To do this, we construct an explicit filling $L$ of $\Lambda_k$ inducing the augmentation $\e_{n;2}$.

Via the resolution procedure, we can assume that the Lagrangian projection $\Pi_{xy}(\Lambda_k)$ is as shown on the left of Figure~\ref{fig:cobordism}. In the language of \cite{EHK}, our filling $L$ is decomposable and is constructed as follows. We concatenate two saddle cobordisms that replace the crossings $a_1$ and $a_3$ in $\Pi_{xy}(\Lambda_k)$ consecutively by their oriented resolution (in the standard knot theory sense). The result is the Legendrian knot $\Lambda'$ shown on the right of Figure~\ref{fig:cobordism}, and the saddle cobordisms yield a cobordism $L_1$ from $\Lambda'$ at the bottom to $\Lambda_k$ at the top. Now by inspection, $\Lambda'$ is a standard Legendrian unknot and thus has a filling $L_2$. The concatenation $L_1 \cup L_2$ is the desired filling $L$ of $\Lambda_k$.

\begin{figure}
\labellist
\small\hair 2pt
\pinlabel $a_1$ at 16 49
\pinlabel $a_3$ at 75 49
\endlabellist
\centering
\includegraphics[width=\textwidth]{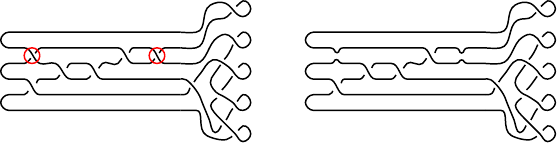}
\caption{\textbf{The cobordism $L_1$ built from two saddle moves.} Resolving the circled crossings on the left via two saddle moves produces the cobordism $L_1$ from a standard Legendrian unknot $\Lambda'$ (right) to $\Lambda_k$. The case $k=2$ is shown.
}
\label{fig:cobordism}
\end{figure}

In constructing the filling $L$, there is one important point to check: in order for the saddle cobordism $L_1$ to be realized as an exact Lagrangian, we need to verify that the Reeb chords $a_1$ and $a_3$ are contractible in the sense of \cite{EHK}. That is, we must find a Legendrian isotopy starting at $\Lambda_k$ such that the Lagrangian projections remain planar isotopic to the knot diagram $\Pi_{xy}(\Lambda_k)$ throughout the isotopy and such that the height of the crossings at $a_1$ and $a_3$ both approach $0$ at the end of the isotopy. To do this, modify $\Pi_{xy}(\Lambda_k)$ by a planar isotopy to obtain the diagram at the top of Figure~\ref{fig:resolution}. This diagram is the Lagrangian projection of a Legendrian knot whose front projection is given by the bottom diagram in Figure~\ref{fig:resolution}: recall that we translate from $xz$ to $xy$ projection by setting $y=dz/dx$. In the front projection, we have not drawn the entire front, but rather just the key portions drawn in color. The remainder of the front, corresponding to the black portion of the Lagrangian projection, is completed as in the resolution procedure from \cite[Proof of Proposition 2.2]{CLI}. That is, we stretch the black portion of the Lagrangian projection of $\Lambda_k$ rightwards in such a way that it coincides with the resolution of the corresponding portion of the front projection (which in turn is given by line segments and small exceptional segments, as in that proof).
It is now apparent that the front can be perturbed by translating the red portion in the negative $z$ direction, in such a way that the two segments labeled $a_1$ and $a_3$ both shrink to a point. This does not change any portion of the Lagrangian projection (up to planar isotopy) apart from $a_1$ and $a_3$ and verifies that $a_1$ and $a_3$ are indeed contractible.

\begin{figure}
\labellist
\small\hair 2pt
\pinlabel $x$ at 26 125
\pinlabel $y$ at 2 150
\pinlabel $x$ at 25 3
\pinlabel $z$ at 2 27
\pinlabel $a_1$ at 43 155
\pinlabel $a_2$ at 144 163
\pinlabel $a_3$ at 220 163
\pinlabel $a_4$ at 55 147
\pinlabel $a_5$ at 66 155
\pinlabel $a_6$ at 122 154
\pinlabel $a_1$ at 43 35
\pinlabel $a_2$ at 144 75
\pinlabel $a_3$ at 219 97
\pinlabel $a_4$ at 54 14
\pinlabel $a_5$ at 66 14
\pinlabel $a_6$ at 122 41
\endlabellist
\centering
\includegraphics[width=\textwidth]{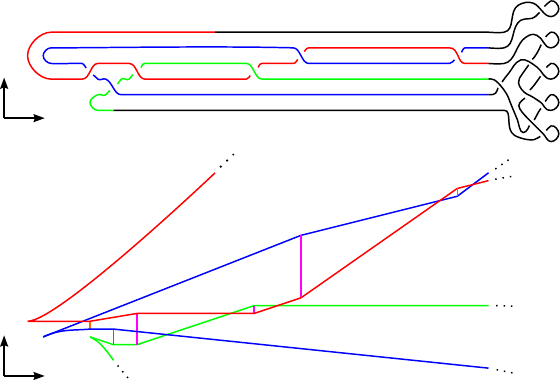}
\caption{The Reeb chords $a_1$ and $a_3$ in the Lagrangian projection of $\Lambda_k$ are simultaneously contractible. Top, the Lagrangian projection for $\Lambda_k$; bottom, a corresponding front diagram for $\Lambda_k$. The omitted portions of the front follow the standard resolution procedure.
}
\label{fig:resolution}
\end{figure}
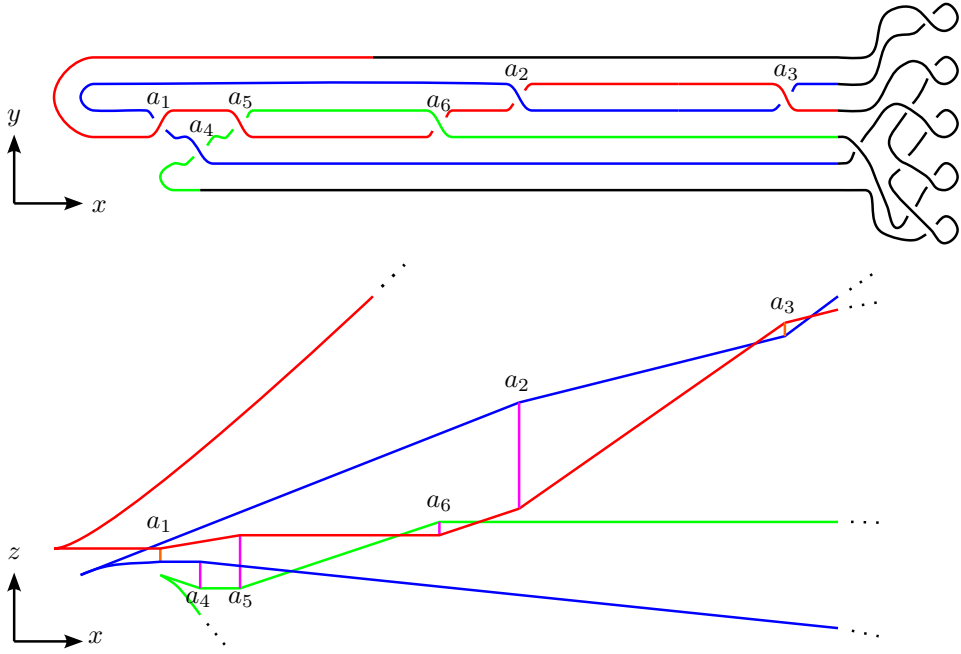

Now let $\e \co \A_{\Lambda_k} \to \Z/2$ be the augmentation to $\Z/2$ induced by the filling $L$. Since the cobordism $L_1$ from $\Lambda'$ to $\Lambda_k$ consists of a concatenation of two saddle cobordisms at contractible Reeb chords, there is a combinatorial formula for the resulting DGA map $\Phi_{L_1} \co (\A_{\Lambda_k},\d) \to (\A_{\Lambda'},\d)$ given by \cite[Proposition~6.18]{EHK}. For our purposes, it suffices to note that $\Phi_{L_1}(a_1)=\Phi_{L_1}(a_3)=1$ since $a_1$ and $a_3$ are the Reeb chords being removed. Since $\e$ is the composition of $\Phi_{L_1}$ with a map $\A_{\Lambda'}\to\Z/2$ corresponding to the filling $L_2$, we conclude that $\e(a_1)=\e(a_3)=1$.

The conditions that $\e(a_1)=\e(a_3)=1$ uniquely determine the augmentation $\e$: from the computation in Section~\ref{sssec:Lambdak-linhom}, any augmentation $\e$ of $\Lambda_k$ satisfies $\e(a_1)+\e(a_3)+\e(a_1)\e(a_2)\e(a_3)=1$, whence $\e(a_2)=1$ in our case, and furthermore $\e$ is uniquely determined by $\e(a_i)$ for $i=1,2,3$. But $\e_{n;2}$ also satisfies $\e_{n;2}(a_1)=\e_{n;2}(a_3)=1$ when $n$ is odd. Hence, $\e=\e_{n;2}$ and thus $\e_{n;2}$ is geometric. This completes the proof of Proposition~\ref{prop:geometric}.

\begin{remark}
A similar but slightly more complicated calculation shows that for the knot $\Lambda_0$ from Section~\ref{ssec:821}, the augmentation $\e_n$ is geometric over $\Z/2$ but not over $\Z$ when $n$ is odd and $n\neq \pm 1$. 
\end{remark}

\subsection{Further comments}
We conclude with some observations about when torsion appears for knots in the Legendrian knot atlas \cite{atlas}.

The knots $m(8_{21})$ and $m(9_{45})B$ are $\Lambda_0$ and $\Lambda_1$ discussed above. (Here we write $m(9_{45})B$ for the second Legendrian representative of $m(9_{45})$ listed in the atlas, and similarly for other knots below.) Direct computation shows that three other knots have torsion: $m(9_{45})A$, $11^n_{95}$, and $11^n_{118}$ (note that $n$ here denotes ``non-alternating''). Like the examples discussed above, for any $m\in\Z$, one can find an augmentation $\epsilon_m$ for each of these knots so that the Legendrian contact homology has a $\Z/m$ summand.

There are also many knots which can be shown not to have torsion:

\begin{definition}
A Legendrian knot is \textit{positive} if it has rotation number $0$ and is Legendrian isotopic to a knot for which all the Reeb chords have nonnegative grading.
\end{definition}

\begin{proposition}
Let $\Lambda$ be a positive Legendrian knot with Thurston--Bennequin number $\tb(\Lambda)$. Then the linearized contact homology for any $\Z$-valued augmentation $\e$ of $\Lambda$ is given by:
\[
\LCH_*^{\e}(\Lambda) \cong \begin{cases} \Z & *=1 \\ \Z^{\tb(\Lambda)+1} & *=0 \\ 0 & \text{otherwise.} \end{cases}
\]
In particular, no linearized contact homology for a positive Legendrian knot can contain torsion.
\end{proposition}
\begin{proof}
  This follows from Sabloff duality. Specifically, by the universal
  coefficient theorem, it suffices to prove that for any field $\k$,
  \[
    \LCH_*^{\e}(\Lambda;\k) \cong \begin{cases} \k & *=1 \\ \k^{\tb(\Lambda)+1} & *=0 \\ 0 & \text{otherwise.} \end{cases}
  \]
  By hypothesis, for $i<0$, $\LCH_i^\e(\Lambda;\k)=0$. So,
  Formula~\eqref{eq:Sabloff-generic} implies that
  $\LCH_i^\e(\Lambda;\k)=0$ for $i>1$ and
  Formula~\eqref{eq:Sabloff-1} implies that
  $\LCH_1^\e(\Lambda;\k)=\k$. Finally, the Euler characteristic of
  $\LCH_*^\e(\Lambda)$ is equal to the writhe minus the number of
  right cusps which, in turn, is $\tb(\Lambda)$, so
  $\LCH_0^\e(\Lambda;\k)$ must have dimension $\tb(\Lambda)+1$.
\end{proof}

Positivity of $\Lambda$ can sometimes be deduced from a presentation of $\Lambda$ as a braid closure.
For instance, it is immediate from the formula for gradings of Reeb chords (see Section~\ref{ssec:background}) that rainbow closures of positive braids (see~\cite{STZ}) are positive; in the atlas, this means $m(3_1),\ m(5_1),\ m(7_1),$ $8_{19},\ 10_{124},\ 15^n_{41185}$ (all torus knots), $10_{139}$, and $12^n_{242}$ are positive. More generally, given an \textit{admissible} positive braid in the sense of~\cite[Definition 2.5]{CN}, the result of multiplying by $\Delta^{-2}$ and taking the closure gives a positive Legendrian; see~\cite[Section 5.1]{CN}. (By~\cite[Proposition 2.7]{CN}, any positive braid containing a half-twist is admissible.) If the topological knot type has a unique representative with maximal Thurston--Bennequin number, and appears in the atlas, then the representative in the atlas must be positive. Alternatively, it is sometimes possible to check directly that a Legendrian knot is positive. Via such considerations, the following knots are also positive:
\begin{center}
  \begin{tabular}{c}
  $m(5_2)B,\ m(7_2)B,\ m(7_2)D,\ 7_3B,\ 7_4, \ m(7_5)B,\ 9_{49}, \ 10_{128}B,$\\
  $m(10_{142})B,\ m(10_{145}),\ m(10_{161}),\ 12^n_{591}$.
  \end{tabular}
\end{center}

This leaves $35$ Legendrian knots in the atlas which admit (graded) augmentations but may or may not have torsion. We gathered some computational evidence that none of these knots admits augmentations with torsion:
\begin{itemize}
\item The dimension of $\LCH^*_\epsilon(\Lambda;\k)$ is the same for all $\ZZ/2$-valued and $\ZZ/3$-valued augmentations $\epsilon$ (where $\k=\ZZ/2$ or $\ZZ/3$). It follows from the universal coefficient theorem that if $\epsilon$ is a $\ZZ$-valued augmentation for which $\LCH^*_\epsilon(\Lambda)$ has 2-torsion or 3-torsion, then in fact $\LCH^*_\epsilon(\Lambda)$ has a copy of $\ZZ/6$.
\item The Bockstein map induced by the short exact sequence $0\to\ZZ/2\to\ZZ/4\to\ZZ/2\to 0$ vanishes for all augmentations for these knots. So, any $2^n$-torsion would have to arise as copies of $\ZZ/4$ rather than $\ZZ/2$ (and so the $6$ in the previous point is, in fact, $12$).
\end{itemize}
(For the $24$ of these knots with crossing number $\leq 9$, the dimension of $\LCH^*_\epsilon(\Lambda;\ZZ/5)$ was also constant, agreeing with the dimension over $\ZZ/2$ and $\ZZ/3$, for all $\ZZ/5$-valued augmentations $\epsilon$, so any 2-, 3-, or 5-torsion would have to arise as $\ZZ/60$-torsion. For more complicated knots, the number of $\ZZ/5$-valued augmentations became too large for our na\"ive programs to check.)

Our limited computations support positive answers to the following
questions:
\begin{question}
  Suppose a Legendrian knot $\Lambda$ has augmentations $\e_1$ and
  $\e_2$ to some field $\k$ so that
  $\LCH_*^{\e_1}(\Lambda;\k)\not\cong\LCH_*^{\e_2}(\Lambda;\k)$. Does it
  follow that $\Lambda$ admits a $\ZZ$-valued augmentation $\e$ so
  that $\LCH_*^{\e}(\Lambda)$ has torsion?
\end{question}

\begin{question}
  Suppose a Legendrian knot $\Lambda$ has a $\ZZ$-valued augmentation
  $\e$ so that $\LCH_*^{\e}(\Lambda)$ has torsion. Does it follow that
  for every prime $p$ there is a $\ZZ$-valued augmentation $\e_p$ so
  that $\LCH_*^{\e_p}(\Lambda)$ contains $p$-torsion? That for every
  integer $n$ there is a $\ZZ$-valued augmentation $\e_n$ so that
  $\LCH_*^{\e_n}(\Lambda)$ has a $\ZZ/n$-summand?
\end{question}

\bibliographystyle{plain}
\bibliography{torsion-biblio}

\begin{thebibliography}{10}

\bibitem{BC:aug-cat}
Fr\'{e}d\'{e}ric Bourgeois and Baptiste Chantraine.
\newblock Bilinearized {L}egendrian contact homology and the augmentation category.
\newblock {\em J. Symplectic Geom.}, 12(3):553--583, 2014.

\bibitem{CN}
Roger Casals and Lenhard Ng.
\newblock Braid loops with infinite monodromy on the {L}egendrian contact {DGA}.
\newblock {\em J. Topol.}, 15(4):1927--2016, 2022.

\bibitem{Chantraine-cobordism}
Baptiste Chantraine.
\newblock Lagrangian concordance of {L}egendrian knots.
\newblock {\em Algebr. Geom. Topol.}, 10(1):63--85, 2010.

\bibitem{CDGG:aug-cat}
Baptiste Chantraine, Georgios Dimitroglou~Rizell, Paolo Ghiggini, and Roman Golovko.
\newblock Noncommutative augmentation categories.
\newblock In {\em Proceedings of the {G}\"{o}kova {G}eometry-{T}opology {C}onference 2015}, pages 116--150. G\"{o}kova Geometry/Topology Conference (GGT), G\"{o}kova, 2016.

\bibitem{Che}
Yuri Chekanov.
\newblock Differential algebra of {L}egendrian links.
\newblock {\em Invent. Math.}, 150(3):441--483, 2002.

\bibitem{atlas}
Wutichai Chongchitmate and Lenhard Ng.
\newblock An atlas of {L}egendrian knots.
\newblock {\em Exp. Math.}, 22(1):26--37, 2013.

\bibitem{DR}
Georgios Dimitroglou~Rizell.
\newblock Lifting pseudo-holomorphic polygons to the symplectisation of {$P\times\Bbb{R}$} and applications.
\newblock {\em Quantum Topol.}, 7(1):29--105, 2016.

\bibitem{EkholmSFT}
Tobias Ekholm.
\newblock Rational symplectic field theory over {$\Bbb Z_2$} for exact {L}agrangian cobordisms.
\newblock {\em J. Eur. Math. Soc. (JEMS)}, 10(3):641--704, 2008.

\bibitem{EES-ori}
Tobias Ekholm, John Etnyre, and Michael Sullivan.
\newblock Orientations in {L}egendrian contact homology and exact {L}agrangian immersions.
\newblock {\em Internat. J. Math.}, 16(5):453--532, 2005.

\bibitem{EES:duality}
Tobias Ekholm, John~B. Etnyre, and Joshua~M. Sabloff.
\newblock A duality exact sequence for {L}egendrian contact homology.
\newblock {\em Duke Math. J.}, 150(1):1--75, 2009.

\bibitem{EHK}
Tobias Ekholm, Ko~Honda, and Tam\'{a}s K\'{a}lm\'{a}n.
\newblock Legendrian knots and exact {L}agrangian cobordisms.
\newblock {\em J. Eur. Math. Soc. (JEMS)}, 18(11):2627--2689, 2016.

\bibitem{Etgu}
Tolga Etg\"{u}.
\newblock Nonfillable {L}egendrian knots in the 3-sphere.
\newblock {\em Algebr. Geom. Topol.}, 18(2):1077--1088, 2018.

\bibitem{EHsum}
John~B. Etnyre and Ko~Honda.
\newblock On connected sums and {L}egendrian knots.
\newblock {\em Adv. Math.}, 179(1):59--74, 2003.

\bibitem{ENsurvey}
John~B. Etnyre and Lenhard~L. Ng.
\newblock Legendrian contact homology in {$\mathbb{R}^3$}.
\newblock In {\em Surveys in differential geometry 2020. {S}urveys in 3-manifold topology and geometry}, volume~25 of {\em Surv. Differ. Geom.}, pages 103--161. Int. Press, Boston, MA, 2022.

\bibitem{ENS}
John~B. Etnyre, Lenhard~L. Ng, and Joshua~M. Sabloff.
\newblock Invariants of {L}egendrian knots and coherent orientations.
\newblock {\em J. Symplectic Geom.}, 1(2):321--367, 2002.

\bibitem{GR}
Honghao Gao and Dan Rutherford.
\newblock Non-fillable augmentations of twist knots.
\newblock {\em Int. Math. Res. Not. IMRN}, (2):1255--1291, 2023.

\bibitem{Golovko}
Roman Golovko.
\newblock On torsion in linearized {L}egendrian contact homology.
\newblock {\em J. Knot Theory Ramifications}, 32(7):Paper No. 2350056, 6, 2023.

\bibitem{Karlsson}
Cecilia Karlsson.
\newblock A note on coherent orientations for exact {L}agrangian cobordisms.
\newblock {\em Quantum Topol.}, 11(1):1--54, 2020.

\bibitem{Leverson}
C.~Leverson.
\newblock Augmentations and rulings of {L}egendrian knots.
\newblock {\em J. Symplectic Geom.}, 14(4):1089--1143, 2016.

\bibitem{MS}
Paul Melvin and Sumana Shrestha.
\newblock The nonuniqueness of {C}hekanov polynomials of {L}egendrian knots.
\newblock {\em Geom. Topol.}, 9:1221--1252, 2005.

\bibitem{NgKCH2}
Lenhard Ng.
\newblock Knot and braid invariants from contact homology. {II}.
\newblock {\em Geom. Topol.}, 9:1603--1637, 2005.
\newblock With an appendix by the author and Siddhartha Gadgil.

\bibitem{NRSSZ}
Lenhard Ng, Dan Rutherford, Vivek Shende, Steven Sivek, and Eric Zaslow.
\newblock Augmentations are sheaves.
\newblock {\em Geom. Topol.}, 24(5):2149--2286, 2020.

\bibitem{CLI}
Lenhard~L. Ng.
\newblock Computable {L}egendrian invariants.
\newblock {\em Topology}, 42(1):55--82, 2003.

\bibitem{Sabloff:duality}
Joshua~M. Sabloff.
\newblock Duality for {L}egendrian contact homology.
\newblock {\em Geom. Topol.}, 10:2351--2381, 2006.

\bibitem{STZ}
Vivek Shende, David Treumann, and Eric Zaslow.
\newblock Legendrian knots and constructible sheaves.
\newblock {\em Invent. Math.}, 207(3):1031--1133, 2017.

\end{thebibliography}

\vspace{11pt}

\end{document}